\documentclass[a4paper,12pt,superscriptaddress]{article}

\usepackage[english]{babel}
\usepackage{amssymb}
\usepackage{amsfonts}
\usepackage{amsmath}

\usepackage{graphicx}

\usepackage[square,comma,numbers]{natbib}

\usepackage[dvips]{color}

\newtheorem{theorem}{Theorem}

\newtheorem{definition}[theorem]{Definition}

\newtheorem{lemma}[theorem]{Lemma}

\newtheorem{proposition}[theorem]{Proposition}

\newenvironment{proof}[1][Proof]{\noindent\textbf{#1.} }{\ \rule{0.5em}{0.5em}}

\newcommand{\nat}{\mathbb N}

\newcommand{\re}{\mathbb R}
\newcommand{\La}{\mathcal L}
\newcommand{\U}{\mathcal U}

\newcommand{\Om}{\Omega}
\newcommand{\lmb}{\lambda}
\newcommand{\Lmb}{\Lambda}
\newcommand{\alp}{\alpha}
\newcommand{\kp}{\kappa}

\begin{document}

\title{\textbf{Popper Functions, Lexicographical Probability, and Non-Archimedean Probability}\thanks{We are indebted to Paul Pedersen, Stanislav Speranski, Sylvia Wenmackers, and two anonymous referees for helpful comments on earlier versions of this article. The first author's research was supported by an ESPRC scholarship. }}
\author{Hazel Brickhill \& Leon Horsten (University of Bristol)}

\maketitle

\begin{abstract}

\noindent Standard probability theory has been extremely successful but there are some conceptually possible scenarios, such as fair infinite lotteries, that it does not model well. For this reason alternative probability theories have been formulated. We look at three of these: Popper functions, a specific kind of non-Archimedean probability functions, and lexicographic probability functions. We relate Popper functions to non-Archimedean probability functions (of a specific kind) by means of a representation theorem: every non-Archimedean probability function is infinitesimally close to some Popper function, and vice versa. We also show that non-Archimedean probability functions can be given a lexicographic representation. Thus Popper functions, a specific kind of non-Archimedean probability functions, and lexicographic probability functions triangulate to the same place: they are in a good sense interchangeable.

 \end{abstract}


\section*{Introduction}

Popper functions were introduced in \cite{Popper}. They have been used quantitatively to describe probabilistic concepts and situations that fall outside the scope of classical probability theory. In particular, Popper functions have proved useful in modelling learning from evidence, indicative and counterfactual conditionals, decision and utility theory \cite{McGee} \cite{Leitgeb}. 

Popper functions are typically defined axiomatically, by a list of properties that these functions are required to satisfy. Of course such presentations are accompanied by demonstrations that the list of axioms is consistent.

It is immediate from the axioms of Kolmogorov that a concept of probability is thereby defined: the Kolmogorov axioms represent intuitive properties of an informal concept of probability. This is not so for the axioms that define Popper functions. These axioms implicitly define a conditional function on a class of propositions. But it is not immediate from the Popper function axioms that they define a concept of conditional \emph{probability}. It is difficult to see which models are allowed by the axioms governing Popper functions. So one may wonder whether there are some Popper functions that really cannot be taken to represent probability. It has been argued that the notion of conditional probability is actually prior to its definition in terms of absolute probability \cite{Hajek 2003}, and so the idea of Popper functions is a reasonable one. But how can we tell that Popper hit upon the right set of axioms?

One principal aim of the present article is to make a case for the thesis that Popper functions do capture a pre-theoretical concept of quantitative probability. We do this by relating Popper functions to a class of functions for which it is more immediate that they capture an informal concept of quantitative probability: we will relate Popper functions to non-Archimedean probability functions.

The idea of infinitely small numbers (infinitesimals) goes back at least to Leibniz and Newton. It was shown to be mathematically coherent in the middle of the last century by Abraham Robinson \cite{Robinson 1961}. 

The concept of fair lotteries on infinite sample spaces motivates the application of infinitesimals to the theory of probabilities. In particular, for such a lottery it seems prima facie reasonable to assign a non-zero but infinitesimally small probability value to  every proposition that states that a given ticket is the winning one. A property that one might want to impose on non-Archimedean probability functions in infinite lottery situations is \emph{regularity}, i.e., the condition that only the impossible event receives probability 0. In this article, we focus on one particular form that such a non-Archimedean probability theory can take, namely the theory that was developed in \cite{NAP} and \cite{W&H}.\footnote{There are of course other approaches to non-Archimedean probability: see for instance \cite{Nelson 1987}. For a defence of the thesis that the theory that we are focussing on in this article constitutes a good framework for developing a theory of infinitesimal probabilities, see \cite{BHW forthc}.}

A key difference between Popper functions and non-Archimedean probability functions is that the former take values in the real $[0,1]$ interval whereas the latter take values in the $[0,1]$ interval of a non-Archimedean extension of $\mathbb{R}$. Yet we will show that Popper functions are closely related to non-Archimedean probability functions.
We will prove a \emph{representation theorem} that relates regular non-Archimedean probability functions to Popper functions (section 3). On the one hand, for every finitely additive Popper function, there is a regular non-Archimedean probability function that is point-wise infinitesimally close to it. On the other hand, for every regular non-Archimedean function there is a Popper function that is point-wise infinitesimally close to it.
The latter is fairly trivial; the former is not.

One direction of the theorem shows that Popper functions cannot do anything that non-Archimedean probability functions cannot do. So if one believes that non-Archimedean probability functions (as explicated in \cite{NAP}) model an intuitive notion of probability well, then Popper functions must too: while they may not give us as much information, they cannot go against our intuitions concerning concepts of probability to any greater extent than non-Archimedean probability functions. The converse direction of the theorem shows that that in all the applications of Popper functions to describe situations or elucidate concepts, the notion of  non-Archimedean probability may also be employed.

There are still reasons one may prefer non-Archimedean probability functions over Popper functions or vice versa: for instance, a non-Archimedean probability function will enable us to compare probabilities at a much finer level than a Popper function, distinguishing between events whose probabilities only differ by infinitesimal amounts. On the other hand, one may argue that this level of detail is uncalled for and not justified by our intuitions. But neither can go too far wrong if you believe the other is correct.

Non-Archimedean probability functions are themselves not completely intuitive. One source of mystery is the non-wellfoundedness of the degrees of infinitesimality that such probability functions entail. This seems to open the prospect that the probability of one event can be smaller than that of another event even though there is no largest degree of infinitesimality at which they differ. This would of course make such functions very hard to picture.

In response to this, we prove a second representation theorem (section 4). This second theorem relates the non-Archimedean probability theories under consideration to lexicographical probability functions as discussed in \cite{Blume et al 1991} and in \cite[sections 1.--3.]{Halpern 2010}. In particular, we show that despite the non-wellfounded structure of the range of non-Archimedean probability functions, probability values can be ordered lexicographically. This phenomenon gives us deeper insight into the structure of non-Archimedean probability functions. Conversely, it links the applications of lexicographic probability theory to non-Archimedean probability. In particular, it suggests ways in which non-Archimedean probability theory can be connected to utility and decision theory by supplying hierarchical orderings of `null events'.\footnote{The connection between utility theory and non-Archimedean probability theory for infinite state spaces is investigated in \cite{Pedersen forthc.}.}

In sum, the representation theorems that are proved in this article are intended so shore up the foundations of all three representations of probability. If one has concerns about any two of the three theories investigated in this article, but not about the third, then the results of this paper offer some reassurance. We have a virtuous triangle of representations: starting anywhere on the triangle, one can get anywhere else.


In the proofs of our representation theorems, we will be building on and extending results in this area that are explained and discussed in \cite{Halpern 2010}. The construction of the non-Archimedean probability functions uses mathematical tools that some readers may not be familiar with ---in particular ultrafilters and ultraproduct constructions. In section 2 where we introduce these functions we give the precise definitions but also describe the basic properties that result, and the short proofs here give a flavour of how these are used. The proofs of the main theorems in the paper are technical, but only an understanding of the statements of the theorems is necessary to grasp their philosophical import. 


\section{Popper functions}

\subsection{Axiomatic description}

Popper functions were originally introduced so that conditional probabilities could be assigned where the event conditioned on would receive absolute probability zero. They are both a weakening and an extension of classical probability. Countable additivity is not imposed (it can be, but then Popper functions become less useful for modelling situations, like the infinite lotteries alluded to earlier, which classical probability cannot model), and additional conditional probabilities are defined. Popper functions are introduced \emph{axiomatically}, the suggestion being that these are the right axioms to model conditional probability. We follow the presentation in \cite{Leitgeb}.

For some propositional language $\La$ let $C:\La \times \La \rightarrow [0,1]$. $C$ is Popper function if it satisfies:
\begin{enumerate}
\item{$C(a,a) = 1$}
\item{If $C(\neg a,a) \neq 1$ then $C(.,a)$ is a probability function (i.e. satisfying the Kolmogorov axioms with finite additivity replacing countable additivity)}
\item{$C(a\wedge b,d)= C(a,d) \times C(b,a \wedge d)$}
\item{If $C(a,b)=C(b,a)=1$ then for all $d\in\La$, $C(d,a)= C(d,b)$}
\end{enumerate}
The intended interpretation of $C(a,b)$ here is the probability of $a$ occurring conditional on $b$ occurring. It is not immediately obvious that conditional probabilities should satisfy these axioms but a little work shows they do, at least in intuitive situations. (If we have $C(\neg a,a) = 1$ then it can be shown that $C(.,a)$ is the constant function 1, so that $a$ is to be interpreted as impossible or contradictory event). Also note that absolute probabilities can be retrieved from a Popper function by taking the probability conditional on a tautology.

Of course none of this would be of any significance if the axioms governing Popper functions were inconsistent. But it can be shown that on all sample spaces, finite and infinite, Popper functions can be defined. 


\subsection{Uses}

Popper functions do solve some of the problems outlined above for the standard theory of probability. In particular, the main selling point of Popper functions is of course reclaiming the conditional probabilities that are intuitively justified. Also, certain situations  that cannot be modelled with standard probability functions (such as a fair lottery on $\nat$), can be consistently modelled using a Popper function \cite{VanFraassen}.

For these reasons, Popper functions have been quite useful in philosophy. In Bayesian theories of belief revision under new evidence, a rational agent will start with a credence function that obeys the standard laws of probability. Then this credence is revised under new evidence $E$, so that the new credence of a proposition is the old credence of that proposition conditional on $E$. But in standard probability theory, if the evidence $E$ has probability zero, such conditionals are undefined. Should we be worried about how to rationally react to evidence that has probability zero? Well, given that this evidence may still be  possible, we should. After all, a theory of rationality is not primarily a practical theory: what really concerns us is what it means to be rational, so we should take all possibilities into account \cite{Levi 1989}. One obvious way of doing this is to use Popper functions.\footnote{For a defence of the role of Popper functions in satisfactory theories of conditionalisation on evidence, see e.g. \cite{Harper 1975}. For a discussion of the limitations of the usefulness of Popper functions for such purposes, see \cite{Pruss 2015}.}

Further uses for Popper functions include theories of indicative and counterfactual conditionals. For an application of Popper functions to the theory of indicative conditionals, the reader may consult {\cite{McGee}; for the use of Popper functions in the theory of counterfactual conditionals, see \cite{Leitgeb}.


However, this usefulness really depends on our being willing to accept that Popper functions provide a good model of conditional probability. And it is not clear that we should be. The problem of probabilities of infinite sets not depending on the probabilities of their component subsets is even worse here than in the standard picture, as we don't even have countable additivity. It seems that the probabilities could somehow be arbitrary, floating free of the probabilities of point-events that make them up. And if so, this cannot be consistent with our intuitive notion of probability. But is that really possible? That is one of the questions this paper seeks to answer, by showing how Popper functions are closely tied to a richer alternative theory: a theory of non-Archimedean probability.


\section{Non-Archimedean probability functions}

From now on, we will focus on one particular non-Archimedean probability, namely the theory that is called $NAP$, which was developed in \cite{NAP}.

In $NAP$, the unconditional probability of an event is defined in terms of the conditional probability of an event. Loosely speaking, the probability $P(A)$ of event $A$ is conceived of as the \emph{limit} of $P(A \mid \lambda)$, for the finite set $\lambda$ ``tending toward infinity'': $P(A)$ is conceived of as the limit of the relative frequency of $A$'s on finite snapshots of the sample space.

In this section, we sketch how $NAP$ functions are constructed, and what their basic properties are.\footnote{For details of the proofs and constructions, see \cite{NAP}. The only difference between the constructions in \cite{NAP} and the constructions in this article is that in the present article, the constructions are carried out using filters instead of ideals. But it is easily verified that the two formulations of the theory are equivalent.}


\subsection{Limits}

Let there be given a sample space $\Omega$. And let $\Lambda$ be a directed (under the inclusion ordering) subset of the collection ${\mathcal{P}}_{fin}(\Omega)$ of finite subsets of $\Omega$ such that $\bigcup \Lambda = \Omega$. For any $\lambda \in \Lambda$ and for any event $A$ (i.e., subset of $\Omega$), $P(A \mid \lambda) \in \mathbb{R}$, where  $P(A \mid \lambda )$ is taken to be defined as in \emph{classical} probability theory using the \emph{ratio formula} (since $\lambda$ is finite):
$$ P(A \mid \lambda) \equiv \frac{\left \vert  A \cap \lambda  \right \vert}{\left \vert \lambda \right \vert} .   $$
Intuitively, ${\mathcal{P}}_{fin}(\Omega)$ contains the finite subsets on which any given event $A$ is ``tested'', and as the test sets move up the ordering, they ``approach infinity''.
We want to conceive of the sought for probability function as a kind of \emph{limit}. 
We are of course interested in the case where $\Omega$  is infinite.

If $\Omega $ is infinite, we take a free ultrafilter $\mathcal{U}_{\Lambda }\ $over $\Lambda $ and we set
\begin{equation}
\mathbb{R^*} \equiv \mathfrak{F}\left( \Lambda ,\mathbb{R}\right) /\mathcal{U}_{\Lambda },
\end{equation}
where $\mathfrak{F}\left( \Lambda ,\mathbb{R}\right)$ is the class of functions from $\Lambda$ to $\mathbb{R}$, and
 $\mathfrak{F}\left( \Lambda ,\mathbb{R}\right) / \mathcal{U}_{\Lambda }$ denotes the set of equivalence classes $\left[ \varphi \right] _{\mathcal{U}_{\Lambda }}\ $with respect to the relation $ \approx _{\mathcal{U}_{\Lambda }}$defined by
\begin{equation*}
\varphi \approx _{\mathcal{U}_{\Lambda }}\psi \Leftrightarrow \exists Q\in \mathcal{U}_{\Lambda },\forall \lambda \in Q,\ \varphi (\lambda )=\psi (\lambda ).
\end{equation*}
}  \label{v1}

The elements of $\mathbb{R^*}$ can then indeed be regarded as `limits':
\begin{definition}[$\Omega$-limit]
\begin{equation}
\lim_{\lambda \uparrow \Omega }\varphi (\lambda ) \equiv \left[ \varphi \right] _{\mathcal{U}_{\Lambda }}.
\end{equation} \label{v2}
\end{definition}

The set $\mathbb{R^*}$ will serve as the range of our Non-Archimedean Probability function. We want to identify $\mathbb{R}$ with a subset of $\mathbb{R^*}$.  Therefore we identify the equivalence class of the function $\varphi _{c}$ with constant value $c$ with the real number $c$. We want to calculate (add, multiply) with the elements of $\mathbb{R^*}$. Therefore we define addition and multiplication on elements of $\mathbb{R^*}$ pointwise. It can then be verified (using standard arguments from non-standard analysis) that $\mathbb{R^*}$ is a (Non-Archimedean) field.


\subsection{Infinitesimal probabilities}

Using this notion of $\Omega$-limit we can now define a non-Archimedean probability function as follows:

\begin{definition}
\begin{equation}
P(A)=\ \underset{\lambda \uparrow \Omega }{\lim }\ P(A \mid \lambda ).
\end{equation}

\end{definition}
The intuitive meaning of  the equation
$P(A)=\ \underset{\lambda \uparrow \Omega }{\lim }\ P(A \mid \lambda )$
is that the probability of an event $A$ is the $\Omega $\emph{-limit} of the conditional probability $P\left( A \mid \lambda \right) $ obtained by a finite sample set $\lambda $. 

It can then be shown that $P$ satisfies the laws for being a finitely additive probability function (except that the value is taken not in $\mathbb{R}$ but in a non-Archimedean field):
\begin{itemize}
\item (NAP1) \textbf{Domain and range. }\textit{The events are \textbf{all} the subsets of }$\Omega $\textit{, which is a finite or infinite sample space. Probability is a total function}
\begin{equation*}
P:\mathcal{P}\left( \Omega \right) \rightarrow \mathbb{R^*}
\end{equation*}
\textit{where }$\mathbb{R^*}$\textit{\ is a superreal field, where a superreal field is an ordered field which contains the real numbers as subfield.} 
\item (NAP2) \textbf{Normalization.}
\begin{equation}
P(\Omega )=1  
\end{equation}

\item (NAP3) \textbf{Additivity.} \textit{If }$A$\textit{\ and }$B$\textit{\ are events and }$A\cap B=\varnothing ,$ \textit{ then}
\begin{equation*}
P(A\cup B)=P(A)+P(B)
\end{equation*}

\end{itemize}



Since the non-Archimedean probability function that is thus defined depends on the initial choice of free ultrafilter $\mathcal{U}$, we should strictly speaking write $P_{\mathcal{U}}$ instead of $P$.


\subsection{Regularity}


We now look at how we can impose a general condition on the limit construction to ensure the probability functions that result are \emph{regular}, i.e., that for such $P$:

\begin{definition}[Regularity]

$\ \forall A\in \mathcal{P}\left( \Omega \right) \setminus \left\{ \varnothing \right\} :$
\begin{equation}
P(A)>0  
\end{equation}

\end{definition}

We have seen that the probability functions $P_{\mathcal{U}}$
that we have defined so far are determined by ultrafilters $\mathcal{U}$ on $\Lambda$ ---which is a directed (under the inclusion ordering) subset of the 
collection $\mathcal{P}_{fin}(\Omega) $. The regularity of $P_{\mathcal{U}}$ can be forced to hold by imposing a condition on the ultrafilter $\mathcal{U}$ on which it is based:

\begin{definition}[fine ultrafilter]
An ultrafilter $\mathcal{U}$ on $\Lambda$ is fine if and only if for every $a\in \Omega$, we have: 
$$ \{ \lambda \in \mathcal{P}_{fin}(\Omega) : a \in \lambda   \}	 \in \mathcal{U}.$$
\end{definition}

Then we immediately have:

\begin{proposition}

If $\mathcal{U}$ is fine, then $P_{\mathcal{U}}$ is regular.

\begin{proof}
Consider any $a  \in \Omega$. Then $P_{\mathcal{U}}(\{a \}) \neq 0$ if and only if $$   \{ \lambda \in \mathcal{P}_{fin}(\Omega): P(\{a \} \mid \lambda) \neq 0 \} \in \mathcal{U}.  $$ But $   \{ \lambda \in \mathcal{P}_{fin}(\Omega): P(\{a \} \mid \lambda) \neq 0 \} = \{ \lambda \in \mathcal{P}_{fin}(\Omega): a \in \lambda  \}, $ which belongs to $ \mathcal{U}  $ by the fineness condition.
\end{proof}

\end{proposition}


\subsection{Weights}

The class of non-Archimedean probability functions that we have defined so far are those that are determined by fine (and free) ultrafilters according to the recipe described in the previous sections. But this class is still too narrow, because it turns out that all those functions $P_{\mathcal{U}}$ are \emph{uniform}:

\begin{proposition}

For all fine ultrafilters $\mathcal{U}$ on $\Lambda$, and for all $a,b \in \Omega$: $$P_{\mathcal{U}}(\{a\}) = P_{\mathcal{U}}(\{ b \}).$$

\begin{proof}
$P_{\mathcal{U}}(\{a\}) = P_{\mathcal{U}}(\{ b \})$ if and only if $$ \{ A \in \mathcal{P}_{fin}(\Omega): P(\{a\} \mid A) = P(\{ b \} \mid A) \} \in \mathcal{U}.  $$
But $ \{ A \in \mathcal{P}_{fin}(\Omega): P(\{a\} \mid A) = P(\{ b \} \mid A)\} = $  \\ $   \{ A \in \mathcal{P}_{fin}(\Omega): P(\{a\} \mid A) \leq P(\{ b \} \mid A) \}  \cap   \{ A \in \mathcal{P}_{fin}(\Omega): P(\{a\} \mid A) \geq P(\{ b \} \mid A) \}  $. Now  $   \{ A \in \mathcal{P}_{fin}(\Omega): P(\{a\} \mid A) \leq P(\{ b \} \mid A) \}= $ \\
$  \{   A \in \mathcal{P}_{fin}(\Omega): b \in A  \}  \cup   \{   A \in \mathcal{P}_{fin}(\Omega): a \not \in A \wedge b \not \in A  \}  .$ By fineness and the superset property for ultrafilters, we indeed have $  \{   A \in \mathcal{P}_{fin}(\Omega): b \in A  \}  \cup   \{   A \in \mathcal{P}_{fin}(\Omega): a \not \in A \wedge b \not \in A  \}  \in \mathcal{U}, $ i.e.,  $   \{ A \in \mathcal{P}_{fin}(\Omega): P(\{a\} \mid A) \leq P(\{ b \} \mid A) \}  \in \mathcal{U}.$ Similarly,  $   \{ A \in \mathcal{P}_{fin}(\Omega): P(\{a\} \mid A) \geq P(\{ b \} \mid A) \}  \in \mathcal{U}$. So their intersection must also be in $\mathcal{U}$. 
\end{proof}

\end{proposition}

But we do not want to build uniformity into the \emph{definition} of our non-Archimedean probability functions. So we allow also probability functions $P_{\mathcal{U}, w}$ that are ``tempered'' by a real-valued \emph{weight function} $w$.
To be more precise, our official definition of non-Archimedean probability functions is as follows:

\begin{definition}[NAP functions]

If $\mathcal{U}$ is a free and fine ultrafilter on a directed set $\Lambda \subseteq \mathcal{P}_{fin}(\Omega)$, and we have a weight function $w: \Omega \mapsto \mathbb{R}^+$, and $P_{\mathcal{U}}$ is as defined in the preceding sections, then $$P(A)=\ \underset{\lambda \uparrow \Omega }{\lim }\ P(A \mid \lambda )$$
is a \textbf{non-Archimedean probability function} (NAP function), where $$P(A \mid \lambda ) \equiv  \frac{\Sigma_{a \in A \cap \lambda}   (w(a))}{\Sigma_  {a \in\lambda}   (w(a))    } .$$ 

\end{definition}

\noindent This then is the class of non-Archimedean probability functions that we will be concerned with in this article. Clearly $P$ is then always a finitely additive probability function, is regular (if $w$ takes strictly positive values everywhere), and may be but need not be uniform.

As before, we should strictly speaking write $     P_{\mathcal{U},w}$ instead of $P$. Often in what follows it will be clear from the context what $w$ and $\mathcal{U}$ are; in such cases, we will omit the subscripts from $     P_{\mathcal{U},w}$.

Our non-Archimedean probability functions $P$ are then determined by a triple $\langle \Omega, \mathcal{U}, w  \rangle$, where

\begin{itemize}
\item $\Omega $ is the sample space;

\item  $\mathcal{U}$ is a free and fine ultrafilter on some directed subset $\Lambda$ of $\mathcal{P}_{fin}(\Omega)$ such that $\bigcup \Lambda = \Omega$;

\item $w:\Omega \rightarrow \mathbb{R}^{+}$ is a weight function.
\end{itemize}
A triple $\langle \Omega, \mathcal{U}, w  \rangle$ is called a \emph{NAP space}.


\subsection{Infinite sums}

The Weierstrass notion of limit allows us to give a rigorous definition of the sum of an infinite sequence. Analogously, the $\Omega $-limit allows the definition of the sum of infinitely many real numbers. In this section we will investigate this operation.

Let $x_{\omega }$ be a family of real numbers indexed by $\omega \in E\subseteq \Omega ;$ the $\Omega $-sum of all ${x_{\omega }}^{\prime }s$ is defined as follows:
\begin{equation}
\Sigma_{\omega \in E}x_{\omega }=\ \underset{\lambda \uparrow \Omega }{\lim } (\Sigma_{\omega \in E\cap \lambda }x_{\omega } )  \label{v4}
\end{equation}
Notice that, since $\lambda $ is always finite, the function
\begin{equation*}
\varphi (\lambda ):=\Sigma_{\omega \in E\cap \lambda }x_{\omega }
\end{equation*}
of $\lambda$ is well defined, yielding always real number as function value. 

Our new type of infinite sum of course differs in certain respects from the usual Weierstrass-sum. First of all, the $\Omega $-sum depends on the choice of a free ultrafilter $\mathcal{U}_{\Lambda }$. This is not the case with the usual series.
So it would actually be more appropriate to write $\Sigma_{\omega \in E;\mathcal{U}_{\Lambda }}x_{\omega }$ rather than $\Sigma_{\omega \in E}x_{\omega }.$ Secondly, 
 the Weierstrass-sum of a series exists only for certain \emph{denumerable} sets of real numbers, while the $\Omega $-sum exists for \emph{every} family of real numbers indexed by $\omega \in E\subseteq \Omega .$ In principle $\Omega \ $ and hence $E$ may have any cardinality. Lastly, the Weierstrass-sum of a series ---if it exists--- is a real number, while the result of a $\Omega $-sum is a hyperreal number in $\mathbb{R^*}$.

Our generalised notion of sum allows us to obtain an analogue of the familiar principle of $\sigma$-additivity. In particular, it can be shown that  \cite[section 3.4]{NAP}:
\begin{proposition}[infinite sum rule]
If $A = \bigcup_{i\in I}A_i$, with $A_i \cap A_j = \emptyset$ for all $i,j \in I$, then:
$$P(A) =  \underset{\lambda \uparrow \Omega }{\lim } (\Sigma_{i\in I} P(A_i \mid \lambda)).$$
\end{proposition}
This natural infinite sum rule holds not only for countable families of events, but for families of events of any cardinality. This means that $NAP$ functions have what is called the \emph{perfect additivity property}.\footnote{For a discussion of the virtues of perfect additivity, see \cite{Seidenfeld et al 2014}.}


The notion of infinite sum also allows us to express the (non-Archimedean) probability of an event in a new way, namely as:

$$     P(A)     = \frac{\Sigma_{a \in A}   (w(a))}{\Sigma_  {a \in\Omega}   (w(a))    }  .   $$


\section{Relating Popper functions to non-Archimedean probability functions}

The main aim of the present section is to prove the following:

\begin{theorem}\label{main theorem}
\begin{enumerate}
\item  For every finitely additive Popper function on any sample space, there is a regular NAP function that is point-wise infinitesimally close to it.
\item For every regular $NAP$ function on any sample space, there is a finitely additive Popper function that is infinitesimally close to it.
\end{enumerate}
\end{theorem}
This theorem extends an earlier theorem that was proved by McGee in \cite{McGee}.

The second part of the theorem that we seek to prove is straightforward. Indeed, it is routine to verify that given any regular non-Archimedean probability function $P$, if $C(B \mid A)$ is defined as $$st(\frac{P(B\wedge A)}{P(A)})$$ (where $st (a)$ denotes the unique real number that is closest to $a$), then $C$ satisfies the axioms governing Popper functions. 

However, the first part of the theorem that we seek to establish is non-trivial. Given a Popper function on some sample space, we will construct a suitable $NAP$ function, but the $NAP$ function will not have the same underlying sample space.

Indeed, the following argument shows that we cannot in general expect the $NAP$ function to have the same sample space.\footnote{Due to Arthur Pedersen, personal communication.}
Consider the following total, regular non-Archimedean (but not $NAP$!) probability function on $\mathbb{N}$:
$$ P_1(A) \equiv \epsilon \cdot \delta_0 (A) + (1- \epsilon) \cdot \sum_{n \in A \cap \mathbb{N}_0}  \frac{1}{2^n} ,$$ 
where $\epsilon$ is some positive infinitesimal and $ \delta_0$ is the $0-1$-valued function such that for every $A \subseteq \mathbb{N}$, $\delta_0(A) = 1$ if $ 0 \in A$, and $\delta_0(A) = 0$ otherwise. Let $C$ be the Popper function $st(P_1)$.
Since $P_1(\{0\} \mid \{ 0,1\})$ is infinitesimally small, $C(\{ 0 \}, \{0,1\})$ must be $0$. Now for every $NAP$ function $P$, $P(\{ 0 \} \mid \{0,1\})$ must be a real number (different from $0$), since $\{0,1\}$ is a finite set. So there can be no $NAP$ function $P$ such that $\left \vert C(\{ 0 \}, \{0,1\}) - P(\{ 0 \} \mid \{0,1\}) \right \vert$ is infinitesimally small.


\subsection{McGee's theorem}\label{McG}

McGee effectively proves the following theorem \cite[p.~181--184]{McGee}:

\begin{theorem}\label{McGee}
For every Popper function for a propositional language $\mathcal{L}$, there is a regular non-Archimedean probability function on $\mathcal{L}$ that is point-wise infinitesimally close to it.
\end{theorem}

The structure of McGee's proof is roughly as follows.

Let $\mathcal{L}$, and a Popper function $C: \mathcal{L} \times \mathcal{L} \mapsto \mathbb{R}$, be given. 

\vspace{0.3cm}

\noindent \textbf{Stage 1}

\vspace{0.1cm}

Let $\mathcal{L}_1,\ldots, \mathcal{L}_i, \ldots$ be an enumeration of the \emph{finitely generated} sub-languages of $\mathcal{L}$. By ``finitely generated language", McGee means a language that is generated by the familiar boolean operations from a finite set of basic events. 
For each $\mathcal{L}_i$, McGee generates a non-Archimedean probability function $Pr_i$ defined on $\mathcal{L}_i$ which is infinitely close to $C$ restricted to $\mathcal{L}_i$.
In other words, he generates a probability function for any finitely generated event-algebra on the sample space $\Omega$. 

Van Fraassen has shown how on the basis of the given Popper function $C$, for every $\mathcal{L}_i$ a finite number $n$ of \emph{ranks} or degrees of infinitesimality are distinguished \cite{VanFraassen}. Every rank $k$ is marked by a `maximally probable' sentence $a_k \in \mathcal{L}_i$ of that rank, so that:
\begin{itemize}
\item $C(b,a_k) >0$ for every sentence $b$ of rank $k$;
\item $C(b,a_k) = 0 $ for every sentence $b$ of rank $>k$;
\item $C(a_k,b) = 0$ for every sentence $b$ of rank $<k$.
\end{itemize}

McGee then shows how in terms of these ranks an infinitesimal probability function $Pr_i$ can be expressed that is infinitely close to $C$.

What McGee's construction shows is effectively that the probability of a sentences $b\in \mathcal{L}_i$ can be represented as: $$  Pr_i(b) = r_0 \cdot  \epsilon^0+\ldots+ r_n \cdot \epsilon^n,   $$ where $\epsilon$ is an infinitesimal number and $r_0,\ldots ,r_n \in [ 0,1 ]$ \cite[section 4.1]{Halpern 2010}. ($r_0$ is then $st[Pr_i(b)]$.)

In other words, $Pr_i(b)$ can be seen as a finite sequence $\langle r_0,\ldots,r_n  \rangle$ of classical probability values. These finite sequences are lexicographically ordered. So the non-archimedean probabilities for $\mathcal{L}_i$ can be given a \emph{lexicographical} representation.

\vspace{0.3cm}

\noindent \textbf{Stage 2}

\vspace{0.1cm}

In the second part of the proof, an ultrafilter on the index set of the $\mathcal{L}_i$'s yields ultraproduct model $Pr$ for the whole of $\mathcal{L}$. By the fact that for each $i,j$, $$Pr_i(b \mid a) \approx C(b,a) \approx Pr_j(b \mid a)  ,$$ we have that $Pr(b \mid a) \approx C(b \mid a)$, and we are done.

\vspace{0.4cm}

Note, incidentally, that the function $Pr$ is not given as a triple $\langle \Omega, \mathcal{U}, w   \rangle$, so the non-Archimedean function that McGee produces is not an $NAP$ function. Also, the function $Pr$ is only finitely additive, and not countably additive.



McGee then reflects on the familiar notion of $\sigma$-additivity \cite[p.~184]{McGee}:
\begin{quote}
For some purposes, it is useful to look at sentential calculi in which we can form infinite conjunctions and disjunctions and probability measures that are countably additive, rather than merely finitely additive. Thus, we may add to the definition of a Popper function this following requirement:
\begin{multline*} \textrm{If } b \textrm{ is the disjunction of } a_0,a_1,a_2,\ldots \textrm{ and if }  a_i \wedge a_j \\  \textrm{ is inconsistent whenever }  i\neq j, 
 \textrm{ then } C(b,c)=\Sigma_{i=0}^{\infty}C(a_i,c) \end{multline*}
For the corresponding condition in terms of nonstandard probability assignments, the first thing that comes to mind is nonsensical:
\begin{multline*}  \textrm{If } b \textrm{ is the disjunction of } a_0,a_1,a_2,\ldots \textrm{ and if } a_i \wedge a_j   \\ \textrm{ is inconsistent whenever } i\neq j, \textrm{ then } P(b)=\Sigma_{i=0}^{\infty}P(a_i)   \end{multline*}
The nonstandard model of analysis will not be topologically complete, and thus this infinite sum will not normally exist. Instead, the condition we require is:
\begin{multline*}  \textrm{If } b \textrm{ is the disjunction of } a_0,a_1,a_2,\ldots \textrm{ and if } a_i \wedge a_j \\  \textrm{ is inconsistent whenever } i\neq j, \textrm{ then } P(b) \approx \Sigma_{i=0}^{\infty}P(a_i)  \end{multline*}
It follows immediately from [theorem \ref{McGee}] that with these two additional requirements, the two approaches will again coincide.
\end{quote}

Of course McGee is right that in a non-Archimedean context, $\sigma$-additivity cannot be the right infinite additivity rule.\footnote{The reasons for this are explained in \cite{W&H}.} But this not a valid reason for retreating to insisting only that the infinite sum (in the Weierstrass sense) of the probabilities and the probability of the infinite union should agree up to an infinitesimal. Indeed, we expect from our theory of probability that the probability of an infinite sum is \emph{exactly} computable on the basis of the probabilities of its components: the probability of an infinite union of a family of events should be expressible as an infinite sum of the probabilities of the events in the family. We have seen in the previous section how, using a generalised limit concept, non-Archimedean probability functions are able to do this.

We have noted that $NAP$ does have a natural infinite additivity rule. Indeed, $NAP$ satisfies \emph{perfect additivity}. So our representation theorem \ref{main theorem} is a substantial strengthening of McGee's theorem \ref{McGee}. We now turn to the proof of the main theorem.


\subsection{Finitely generated languages}

In this section we start, as McGee does \cite[p.~181--183]{McGee}, by proving the result for finitely generated languages, and then in the next section we will use this to construct an $NAP$ space for an arbitrary language (giving part 1 of Theorem \ref{main theorem}, and so completing our main theorem). The construction of the $NAP$ function in this section is fairly hands on, and mathematically not very difficult.

First we look at the language which describes the events. Let $\mathcal L_f$ be a finitely generated propositional language. So $\mathcal L_f$ has a finite number of atomic propositions, and is closed under the Boolean operations of conjunction, disjunction and negation.
As $\mathcal L_f$ is finitely generated we can choose $b_1, b_2, \dots ,b_m$ from $\La_f$ such that:
\begin{enumerate}
\item $b_i\wedge b_j \leftrightarrow \bot$ for all $i,j\leq m$ with $i\neq j$ 
\item $\top \leftrightarrow b_1\vee b_2\vee \dots \vee b_m$
\item For any sentence $a$ of $\mathcal L_f$ which is not a contradiction, there exists $\{i_1,\dots i_k\} \subseteq \{1,2,\dots,m\}$ such that $a \leftrightarrow b_{i_1} \vee \dots \vee b_{i_k}$
\end{enumerate}
So the $b_i$'s are the most fine-grained description of the ``state of the world'' that our language can give. Let us call these the \emph{normal atoms} of the language (noting they are not necessarily atomic propositions). We call a sentence \emph{normal} if it is the disjunction of $b_i$'s; it is clear from the above argument that each sentence is equivalent to a normal sentence.

Now suppose we have a Popper function $C:\mathcal L_f \times \La_f \rightarrow \re$. For any sentences $a,b$ of $\La_f$, $C(a,b)$ is to be interpreted as the probability of $a$ conditional on $b$. We will assume that the Popper function is \emph{regular} in the sense that the only $x$'s such that $C(.,x)$ is the constant function 1 are contradictions (and any atomic proposition is not a contradiction). However, this is not an important constraint as all such sentences that are not contradictions can still be interpreted as the empty event, which will receive probability zero.

We now wish to construct a $NAP$ space with probability function $P$ which will agree with the Popper function $C$ up to an infinitesimal difference. To define an $NAP$ space we have seen that (section 2) it is sufficient to assign a sample space $\Om$, a weight function $w:\Om\rightarrow\re^+$ and a directed set $\Lmb$ of finite subsets of $\Om$.  In order to compare a Popper function defined on a propositional language (where events are just propositions) to an $NAP$ function defined on a sample space (where events are subsets of that sample space) we must also specify the interpretation of the propositions within the sample space. Then for any sentences $a,b$ of $\La_f$ if $\bar{a},\bar{b}\subseteq\Om$ are their interpretations we  want $P$ to be such that $st(P(\bar{a}|\bar{b}))=C(a,b)$.  Contradictions will be interpreted as the empty set, and tautologies as the full sample space, but for the details we first define the sample space $\Om$. 

We follow McGee's construction at first (\cite[p.~182]{McGee}), although in slightly different terms. Define $a_0$ to be the normal equivalent of a tautology. Set $a_{k+1}$ to be the disjunction of all $b_i$'s such that $b_i$ logically entails $a_k$ and $C(b_i,a_k)=0$. Set $rk(C)$ to be the largest $n$ such that $a_n$ is defined (i.e. where there are normal atoms $b_i$ with $C(b_i,a_{n-1})=0$). To see that $rk(C)$ is well defined, note that as there are only finitely many normal atoms, and Popper functions are finitely additive, each $a_k$ must be the disjunction of strictly fewer normal atoms than the previous one. We call $rk(C)$ the \emph{rank} of the Popper function \cite{VanFraassen} and also set, for each $b_i$, $rk(b_i)$ to be the least $k$ such that $C(b_i,a_k)>0$.


We now depart from McGee to carry out the construction of the $NAP$ space. Importantly, as mentioned  above, we need to give a sample space which is not simply the normal atoms of the language. Instead we will interpret each normal atom by an infinite subset of the sample space. We also specify the weight function and a directed set, and this triple will then yield the $NAP$ function. After this, we need to check that the probabilities assigned agree with the Popper function up to infinitesimals. The construction of the $NAP$ function is fairly straightforward, and it is intuitive to see why it will work, though the proof is not quick. \\

\textbf{Sample space:}
We set $\Om=\aleph_{0}$\\

\textbf{Interpretation:}
 For the interpretation of $\La_f$ it is enough to assign each $b_i$ to a subset of $\Om$, as then the interpretation of any proposition $a$ of $\La_f$ will just be the union of the interpretations of each $b_i$ that compose the normal equivalent of $a$. To do this we assign each $b_i$ to a set $\bar{b}_i$ such that $|\bar{b}_i|=\aleph_0$, each $\bar{b}_i$ is disjoint from the other $\bar{b}_i$'s, and $\bigcup \{\bar{b}_i: i\in\{1,2,\dots m\}\}=\Om$ (so all the points in the sample space are assigned to some $b_i$).\\

\textbf{Weight function:}
We define $w:\Om\rightarrow\mathbb{R}^+$, by
 $$w(x)=C(b_i,a_k)\text{ where }x\in\bar b_i \text{ and }k=rk(b_i)$$\\
Note that this is well defined as for $x\in\Om$ there is a unique $b_i$ such that $x\in\bar{b}_i$, and by definition of $rk(b_i)$ we have $w(x)>0$.\\

\textbf{Directed set:} To show agreement with the Popper function, we will use a directed set $\Lmb$ such that all $ \lmb\in\Lmb$ satisfy the following two properties:
\begin{eqnarray}
\
 \forall i,j\in\{1,\dots m\} \, \, rk(b_i)=rk(b_j) \rightarrow |\bar{b}_i\cap\lmb|=|\bar{b}_j\cap\lmb| \label{lambda} \label{lambda}\\
 \forall i,j\in\{1,\dots m\}\, \, rk(b_i)<rk(b_j) \rightarrow |\bar{b}_i\cap\lmb|>|\bar{b}_j\cap\lmb|^{2} \label{lambda2}
\end{eqnarray}

\noindent Property \ref{lambda} will ensure that conditional probabilities using propositions that are all of the same rank agree with the Popper function (in fact they will be the same, not just infinitesimally close), and property \ref{lambda2} will ensure that propositions with a lower rank will dominate, so those of a higher rank only make an infinitesimal difference.
\begin{lemma}\label{directedset}We can find $\Lmb\subset\mathcal{P}_{fin}(\Om)$ such that properties (\ref{lambda}) and (\ref{lambda2}) hold for every $\lmb\in\Lmb$.
\end{lemma}
\begin{proof}
This is easy as we can just take $$\Lmb=\{\lmb\in \mathcal{P}_{fin}(\Om): \lmb \text{ satisfies (\ref{lambda}) and (\ref{lambda2}})\}$$
Then $\Lmb$ is directed: any $x\in \Om$ can be incorporated into such a set so $\bigcup\Lmb=\Om$, and for any $\lmb_1,\lmb_2\in\Lmb$ it is straightforward to expand $\lmb_1\cup\lmb_2$ to satisfy (\ref{lambda}) and (\ref{lambda2}).
\end{proof}
\bigskip

Now we have defined an $NAP$ space (up to the choice of ultrafilter $\U_\Lmb$, but any choice will have the properties we need) and so we have specified an $NAP$ function $P:\mathcal{P}(\Om)\rightarrow \re ^*$. So now we need to show that $P$, the $NAP$ function, agrees with $C$, the Popper function, up to the standard part. This takes rather more work than one might expect, becoming fairly technical but the intuition is in \ref{lambda} and \ref{lambda2} above.

We start with a simple case where we are just dealing with normal atoms of the same rank.\footnote{Henceforth we just use, e.g. $b$ for $\bar{b}$, which will be unambiguous.} 
\begin{lemma} \label{st} For any normal atom $b_i$ and $b=b_i\vee b_{j_1}\vee\dots\vee b_{j_k}$ with $rk(b_i)=rk(b_{j_1})=\dots=rk(b_{j_k}) $ we have 
\begin{equation} P(b_i|b)=C(b_i,b)
\end{equation}
\end{lemma}
So in this case the value given by the $NAP$ function is actually the same as that given by the Popper function, not just infinitesimally close.

\begin{proof}
We have for all $\lmb\in\Lmb$ 
\begin{eqnarray*}
\frac{P(b_i\cap\lmb)}{P(b\cap\lmb)} =\frac{\sum_{x\in{b}_i\cap\lmb} w(x)}{\sum_{x\in {b}\cap\lmb}w(x)}  & \mbox{by definition of $P$} \\ 
=  \frac{|b_i\cap\lmb| C(b_i,a_{rk(b_i)})}{\sum_{b_j\in b} |{b}_j\cap\lmb| C(b_j,a_{rk(b_i)})} & \mbox{by definition of $w$} \\ 
=  \frac{C(b_i,a_{rk(b_i)})}{\sum_{b_j\in b} C(b_j,a_{rk(b_i)})} &\mbox{by (\ref{lambda})} \\
=  \frac{C(b_i,a_{rk(b_i)})}{C(b,a_{rk(b_i)})}=C(b_i,b) &\mbox{by (2) and (3) of Popper functions.}
\end{eqnarray*}
This property is then preserved in the $\Omega$-limit.
\end{proof}

\bigskip
We are now ready to show we have agreement up to an infinitesimal in the general case:
\begin{theorem}
For any $a,b\in\La$ we have
\begin{eqnarray}
C(a,b)=st(\frac{P(\bar{a} \cap \bar{b})}{P(\bar{b})})\label{standard agreement}
\end{eqnarray}
\end{theorem}

\begin{proof}
Fix $a,b\in\La$. We have $$P(a|b)=\frac{P(a\cap b)}{P(b)}=\frac{\sum_{b_i\in a\cap b}P(b_i)}{P(b)}= \sum_{b_i\in a\cap b} \frac{P(b_i)}{P(b)} $$ as $P$ is finitely additive.\\

Let $d$ be the minimum rank of all the normal atoms in $b$. We will see that only the normal atoms of this rank are relevant for the standard part. Set $b_d$ to be the union of all the normal atoms $b_j$ of rank d and $b_{>d}$ to be the union of all those with rank greater than $d$, so $b=b_d\cup b_{>d}$. Now we have for $b_i$ in $b$:  
$$P(b_i|b) = \frac{P(b_i)}{P(b_d)}\frac{P(b_d)}{P(b)}$$
Now for any $\lmb\in\Lmb$, taking $k$ to be the number of $b_i$'s in $b_{>d}$,  by property (\ref{lambda2}) of $\Lmb$ we have$$k^2 |{b}_d\cap\lmb|>|{b}_{>d}\cap\lmb|^{2} \Rightarrow k^2 |{b}\cap\lmb|>|{b}_{>d}\cap\lmb|^{2}$$

Now by property (\ref{lambda}) and the definition of the weight function we have: 
$$ \frac{P({b}_{>d}\cap\lmb)}{P({b}\cap\lmb)}=\frac{|b_{>d}\cap\lmb|\sum_{{b}_i\in b_{>d}} C(b_i,a_{rk(b_i)})}{|b_{d}\cap\lmb|\sum_{b_i\in {b}}C(b_i,a_{rk(b_i)})}$$
but $$M=\frac{\sum_{{b}_i\in b_{>d}} C(b_i,a_{rk(b_i)})}{\sum_{b_i\in {b}}C(b_i,a_{rk(b_i)})}$$ is constant for any $\lmb$ so 
$$\frac{P({b}_{>d}\cap\lmb)}{P({b}\cap\lmb)}=M\frac{|b_{>d}\cap\lmb|}{|b_d\cap\lmb|}<Mk^2\frac{|b_{>d}\cap\lmb|}{|b_{>d}\cap\lmb|^2}=Mk^2\frac{1}{|b_{>d}\cap\lmb|}$$
Thus we have for any  $n\in\nat$ we can choose $\mu\in\Lmb$ such that: $$ \,\,\, \forall \lmb\in \Lambda \text{ with } \lmb\supseteq\mu \,\,\, \frac{P({b}_{>d}\cap\lmb)}{P({b}\cap\lmb)}<1/n$$ 

 By preservation to the $\Lmb$ limit (fineness) this implies $st(\frac{P(b_{>d})}{P(b)})=0$ and so $st(\frac{P(b_{d})}{P(b)})=1$ by standard laws of probability. 
 
Then $$st(P(b_i|b))=st \frac{P(b_i)}{P(b_d)}\frac{P(b_d)}{P(b)}=st \frac{P(b_i)}{P(b_d)}st\frac{P(b_d)}{P(b)}=st(P(b_i|b_d))$$

Going back to the general case, 
$$P(a|b)=\sum_{b_i\in a\cap b} \frac{P(b_i)}{P(b)}=\sum_{b_i\in a\cap b_d} \frac{P(b_i)}{P(b)}  \,\,\,+ \sum_{b_i\in a\cap b_{>d}} \frac{P(b_i)}{P(b)}$$
and the latter term is infinitesimal, as it is less than $\frac{P(b_{>d})}{P(b)}$. Now we can apply lemma \ref{st}: taking the standard part (and noting the sums here are all finite) we get
 $$st(P(a|b))= st\sum_{b_i\in a\cap b_d}\frac{P(b_i)}{P(b)} = st \sum_{b_i\in a\cap b_d}\frac{P(b_i)}{P(b_d)} = \sum_{b_i\in a\cap b_d} C(b_i,b_d)$$
 Now $C(b_i,b_d)=C(b_i,b)$ as by axiom 3. of Popper functions $C(b_i,b)=C(b_i,b_d)C(b_d,b)$ and $C(b_d,b)=1$. So we must have:  $$\sum_{b_i\in a\cap b_d} C(b_i,b_d)=\sum_{b_i\in a\cap b_d}C(b_i,b)=C(a,b)$$
 Putting this all together we get $$st(P(a|b))=C(a,b)$$ as required.
\end{proof}\\
Thus part 1 of Theorem \ref{main theorem} is proved.

We now describe a slight extension of this construction which we will use to construct an $NAP$ space for an infinite language. Let $\La_i$ be a finitely generated language with $\{b_1\dots b_m\}$ the normal atoms and $(\Om_i, w_i, \Lambda_i)$ be the $NAP$ space generated as above. Let $\bar{a}^i\subset \Om_i$ be the interpretation of the proposition $a$ of $\La_i$. Suppose we have a different language $\La_p$ (which may not be finitely generated, and may include some of the propositions of $\La$) interpreted on a domain $\Om_p$ (which may be uncountable), and let  $\bar{a}^p\subset \Om_p$ be the interpretation of the proposition $a$ of $\La_p$. Suppose we also have a weight function $w_p:\Om_p\rightarrow\mathbb{R}^+$. 

We take $\La$ to be the language generated by the atomic propositions of $\La_i$ and $\La_p$ augmented by the atomic proposition $p$ not in either $\La_i$ or $\La_p$ ($p$ stands for `previous'), and define a new $NAP$ space as follows. We set $\Om=\Om_i\cup\Om_p$ (assuming $\Om_i$ and $\Om_p$ are disjoint, otherwise we could make them so be introducing an index), and the give the elements the same weighting as before so $w(x)=w_i(x)$ if $x\in\Om_i$ and $w(x)=w_p(x)$ if $x\in\Om_p$. We interpret $p$ as $\Om_p\subset\Om$ and any other atomic $a$ in $\La$ as $\bar{a}^i \cup\bar{a}^p$ (where $\bar{a}^x=\emptyset$ if $a\notin \La_x$).
Note that for each of the normal atoms $b_j$s of $\La_i$ there is a sentence $b_j'=b_i\wedge\neg p$ of $\La$ which has the interpretation $\bar{b}_j$ as before.

 We want $st(P(p|a))=0$ if $a$ is a non-contradictory proposition of $\La_f$, so we treat $p$ as having a rank greater than any of the normal atoms of $\La_i$:  so we can define a directed set $\Lambda$ on $\Om$ in the same way as in Lemmas \ref{directedset}, such that $\Lmb$ satisfies (\ref{lambda}) and (\ref{lambda2}) as above, and in addition $|\bar{b}_j\cap\lmb|>|\bar{p}\cap\lmb|^{2}$ for any normal atom $b_j$ of $\La_i$.

Let $P$ be the $NAP$ function defined by $(\Om, w,\Lambda)$. It then follows that for any $a$ and $b$ in $\La_i$ we have $st(P(a|b))=C(a,b)$. To see this we can just treat $p$ as an extra normal atom and the argument before will go through, as any $a$ or $b$ in $\La_i$ that is not a contradiction will be implied by one of the $b_i$'s, so  $st(P(a))=st(P(a\wedge\neg p))$ and $st(P(b))=st(P(b\wedge\neg p))$. Thus: $$st(P(a|b))=st(\frac{P(a\wedge p)+P(a\wedge \neg p)}{P(b\wedge p)+P(b\wedge \neg p)})=st(\frac{P(a\wedge \neg p)}{P(b\wedge \neg p)})=C(a,b)$$


\subsection{Infinitely generated languages}

In this section we demonstrate how to extend the above result to an arbitrary language $\La$. McGee uses an ultraproduct construction to do this, but his straightforward approach\footnote{The ultraproduct construction in itself is not exactly simple, but its application is well developed, and a standard technique in model theory.} will not work here. The problem is that taking an ultraproduct will yield a weight function taking values in a non-standard extension of the real numbers, which is not how $NAP$ works.\footnote{There is no prima facie reason why $NAP$ should not be extended to allow for this, after all the same additivity principle could be introduced in such a space, so that the probability values would take values in a further extension of the real numbers. However, how we could get the infinite additivity properties of $NAP$ in the ultraproduct is not obvious.} The construction given here is rather different, although ultraproduct techniques are employed. Readers who are not familiar with ultrafilters and ultraproducts may skip this proof, as it is somewhat technical, and the basic idea is really the same is taking an ultraproduct: we combine, with some interesting mathematics, the models for finitely generated sublanguages into a model for the whole language.
\bigskip

Let $\La$ be an infinitely generated propositional language and $C$ a Popper function on $\La$.  Let $\kappa=|\La|$ and $\langle i_\alpha:\alpha<\kappa\rangle$ be a wellordering of all finite systems of atomic sentences from $\La$. For each $\alpha<\kappa$ let $\La_{i_\alp}$  be the language generated by the atomic propositions in $i_\alp$, so each finitely generated sublanguage of $\La$ is just some $\La_{i_\alpha}$.
 \\
 It is easy to see that for every for every finitely generated sublanguage $\La_{i_\alp}$ we can construct is an $NAP$ space $(\Om_{i_\alp},w_{i_\alp},I_{\Lmb_{i_\alp}})$ which satisfies \emph{standard part agreement} (\ref{standard agreement}) - we simply apply the method of the previous section for the language $\La_{i_\alp}$ and the Popper function $C$ restricted to this sublanguage.
We now build a new $NAP$ space from these.\\

\textbf{The sample space:} set 
$$\Om=\{\langle x,\alp\rangle :x\in\Om_{i_\alp}\}=\bigcup_{\alp<\kp} (\Om_{i_\alp}\times\{\alp\})$$ 
Essentially $\Om$ is just the union of all the $\Om_i$'s - the ordered pair construction is just to make sure the elements for different sublanguages are kept distinct. \\

\textbf{The interpretation:} 
For any proposition $a$ in $\La$ the interpretation of $a$ is $\{\langle x,\alp \rangle: a\in\La_{i_\alp} \wedge x\in\bar{a}_{i_\alp}\}$ where $\bar{a}_i$ is the interpretation of $a$ in our model for $\La_i$. Note that this is consistent with $\Om$ being the interpretation of a tautology.\\

\textbf{The weight function:} For $\langle x,\alp\rangle\in\Om$ set 
$w(\langle x,\alp\rangle)=w_{i_\alp}(x)$.
Clearly this is well defined on all of $\Om$ and always strictly positive.\\

\textbf{A directed set?} 

Here we must switch our approach from simply using a directed set, where any ultrafilter including that directed set will give the desired properties, to actually using the ultrafilter directly. \\

For each $\alpha<\kappa$ we will define an $NAP$ space $(\Om_{\alp},w_{\alp},{\Lmb_{\alp}})$, distinct from $(\Om_{i_\alp},w_{i_\alp},{\Lmb_{i_\alp}})$ above, except for in the first case where we set $(\Om_{0},w_{0},{\Lmb_0})=(\Om_{i_0},w_{i_0},I_{\Lmb_{i_0}})$. For $\alp>0$ we construct the $NAP$ space as in the extension at the end of the previous section. Take $\La_i$, $\Om_i$ etc. there as $\La_{i_\alp}$, $\Om_{i_\alp}\times\{\alp\}$ etc.

We take $\La_{p_\alp}$ etc. to encompass the previous finite sublanguages, being the language generated by all the atomic propositions of the $\La_{i_\beta}$ for $\beta<\alp$. Set $\Om_{p_\alp}=\bigcup_{\beta<\alp}\langle \Om_{i_\beta}\times\{\beta\}\rangle$, the interpretation of any $a$ in $\La_{p_\alp}$ just as on $\Om$ but restricted to $ \Om_{p_\alp}$, and $w_{p_\alp}=w\restriction \Om_{p_\alp}$. 

Then following the earlier construction we have $\Om_\alp=\Om_{i_\alp}\cup\Om_{p_\alp}= \bigcup_{\beta\leq\alp}\langle \Om_{i_\beta}\times\{\beta\}\rangle\subset\Om$ and $w_\alp=w\restriction\Om_\alp$. Set $\Lmb_\alp$ to be the directed set as constructed there and take  $\U_\alp$ to be a corresponding ultrafilter. Note that the probability function $P_\alp$ generated by the $NAP$ space $(\Om_\alp, w_\alp,\U_\alp)$ will satisfy \emph{standard part agreement} (\ref{standard agreement}) for any propositions from $\La_{i_\alp}$, although not necessarily for all propositions in $\La_{i_\alp}\cup\La_{p_\alp}$.\\

We construct the ultrafilter for our full $NAP$ space from these $\U_\alp$ together with an ultrafilter on $\kp$.
Let $\tilde{\U}$ be an ultrafilter on $\kp$ such that for each atomic proposition $a$ from $\La$, $\{\alp\in\kp: a\in \La_{i_\alp}\}\in\tilde{\U}$ 
\footnote{This is possible as such sets have the \emph{finite intersection property} and so can be extended to an ultrafilter \cite{hyperreals}.}
and note that all \emph{end-seqments} $\{\alp\in\kp:\alp>\beta\}$ are in $\tilde{\U}$ . This is because, as $\La$ has $\kp$ many atomic propositions, for any $\beta<\kp$ there is some proposition $a\in\La$ which is not in any $\La_{i_\alp}$ for $\alp<\beta$.\\

Now we can define the ultrafilter $\U$ on $\mathcal{P}_{fin}(\Om)$ which we will use for our $NAP$ space. First to ease notation, for $X\subset\mathcal{P}_{fin}(\Om)$ set $X_\alp:=X\cap\mathcal{P}_{fin}(\Om_\alp)$.
Define $\U \subset \mathcal P(\mathcal{P}_{fin}(\Om))$ by $$X\in\U\leftrightarrow \{\alp:X_\alp\in\U_\alp\}\in\tilde{\U}$$
\begin{lemma}
$\U$ is a non-principal ultrafilter, and $\U$ is fine.
\end{lemma}
\begin{proof}
\begin{enumerate}
\item \textbf{supersets} \\Let $Y\subseteq\Om$, $X\subseteq Y$. Then we have $Y_i\supseteq X_i$ for all $\alp$ so $$X\in\U \Rightarrow \{\alp:X_\alp\in\U_\alp\}\in\tilde{\U} \Rightarrow \{\alp:Y_\alp\in\U_\alp\}\in\tilde{\U} \Rightarrow Y\in\U$$
where the first and third implications are simply by definition, and the second as each $\U_\alp$ is an ultrafilter.

\item \textbf{intersection} \\Let $X,Y\in\U$. We have $(X\cap Y)_\alp=X_\alp\cap Y_\alp$ so:
$$ \{\alp:X_\alp\in\U_\alp\}, \{\alp:Y_\alp\in\U_\alp\}\in\tilde{\U} \Rightarrow \{\alp:X_\alp\in\U_\alp\wedge Y_\alp\in\U_\alp\}\in\tilde{\U} $$ $$\Rightarrow  \{\alp:X_\alp\cap Y_\alp\in\U_\alp\}\in\tilde{\U} \Rightarrow X\cap Y\in\U$$
where the first implication is because $\tilde{\U}$ is an ultrafilter and the second because all the $\U_\alp$'s are.

\item \textbf{ultra}\\ First note that $\Om_\alp\setminus X_{\alp}=(\Om\setminus X)_\alp$. Thus
 $$X\in\U \Leftrightarrow \{\alp:X_\alp\in\U_\alp\}\in\tilde{\U} \Leftrightarrow \{\alp:X_\alp^c\notin\U_i\}\in\tilde{\U}$$ $$ \Leftrightarrow \{\alp:X_\alp^c\in\U_\alp\}\notin\tilde{\U} \Leftrightarrow X^c\notin{\U}$$

\item \textbf{non-principal} \\ It is enough to show that no finite set is in $\U$ (see \cite[p.~38]{hyperreals}). But this is clear as $X$ finite $\Rightarrow \{\alp:X_\alp\neq\emptyset\}$ is finite $\Rightarrow X\notin \U$.

\item \textbf{fine} Let $x\in \Om$. Then there is some $\alp\in\kp$ such that $x\in\Om_{i_\alp}$, so then $x\in\Om_\beta$ for all $\beta\geq\alp$. Let $X=\{y\in \mathcal{P}_{fin} (\Om):x\in y\}$. Then $X_\beta=\{y\in\mathcal{P}_{fin}(\Om_\beta):x\in y\}$ so as each $\U_\beta$ is fine, for all $\beta\geq\alp$ we have $X_\beta\in\U_\beta$ and thus $X\in\U$.
\end{enumerate}

\end{proof}

Now we have defined an $NAP$ space $(\Om,w,\U)$, so we set $P:\mathcal{P}_{fin}(\Om)\rightarrow \mathbb{R}^*$ to be the corresponding $NAP$ function. It remains to show that for any $a,b$ from $\La$ we have $st(P(a|b))=C(a,b)$.

To do this we exploit an alternative interpretation of the $NAP$ function generated from above. In the section on $NAP$ we defined $P(a)$ to be the ultrafilter equivalence class of sequence $\langle P(a|\lmb):\lmb\in \mathcal{P}_{fin}(\Om)\rangle$, but we can also think of $P(a)$ as the ultraproduct of $P_\alp(a)$ under the ultrafilter $\tilde\U$. To see this is equivalent, first we need to see that the hyperreal fields generated by the two processes are the same. But this is straightforward: 
For any functions $\varphi, \psi:\mathcal{P}_{fin}(\Om)\rightarrow\re$ we have:
\begin{eqnarray*}
\varphi \approx _{\mathcal{U}}\psi \Leftrightarrow \exists Q\in \mathcal{U},\forall \lambda \in Q,\ \varphi (\lambda )=\psi (\lambda )\\
\Leftrightarrow \exists \tilde{Q}\in \tilde{\mathcal{U}},\forall \alp\in \tilde{Q},\exists X \in \U_\alp, \forall\lambda \in X,\ \varphi (\lambda )=\psi (\lambda)\\
\Leftrightarrow \exists \tilde{Q}\in \tilde{\mathcal{U}},\forall \alp\in Q, \varphi \approx _{\mathcal{U}_\alp}\psi 
\Leftrightarrow [\varphi]\approx_{\tilde{\U}} [\psi]
\end{eqnarray*}
where we take $[\varphi](\alp)=[\varphi]_{\U_\alp}$, so $[\varphi]$ is the function taking $\alp$ to the equivalence class of $\varphi$ under $\U_\alp$.
It is not hard to see that this reasoning shows the hyperreal fields generated by the two processes are isomorphic under the obvious mapping. Thus it makes sense to ask whether $$P(a)=[P_\alp(a)]_{\tilde{\U}}.$$ This does in fact hold as by definition of $\tilde{\U}$ we have $\{\alp:a\in\La_{i_\alp}\}\in\tilde{\U}$ and for any $\alp$ with $a$ in $\La_\alp$ the definition of the $NAP$ function gives $P_\alp(a|\lmb)=P(a|\lmb)= \frac{\Sigma_{x \in a \cap \lambda}   (w(x))}{\Sigma_  {x \in\lambda}   (w(x))    }$, so these are the same function on a set in the ultrafilter $\tilde{\U}$.

We have gone to all this trouble because in order to show $P$ agrees with the Popper function $C$ we will use the fact that the functions $P_\alp$ agree with $C$, and then, using the latter presentation, apply \L os' Theorem, which states that any first-order sentence $\phi$ that holds in ultrafilter-many models will also be true in the ultraproduct model. Now for the detail.

Fix $a,b$ in $\La$ and set $r=C(a,b)$. Let the first order (not propositional) language $\La'$ include the language of real analysis plus constant terms `$p_{a\wedge b}$', `$p_{b}$' and `$t_{r}$', and the unary  predicate $N$. We define models $\mathcal{M}_\alp$: set the domain $D_\alp$ to be the nonstandard reals produced by the ultrafilter $\U_\alp$. Give the language of real analysis its normal interpretation over the hypereals, and set $\mathfrak{I}_\alp(t_r)=r$, $\mathfrak{I}_\alp(N)=\nat$ (the true natural numbers, not their non-standard extension under $\U_\alp$) and where $a,b\in \La_{i_\alp}$ set $\mathfrak{I}_\alp(p_{a\wedge b})=P_\alp(a\cap b)$ and $\mathfrak{I}_\alp(p_b)=P_\alp(b)$ and be arbitrary otherwise.\
 Now for each $\alp$ such that $a,b\in\La_{i_\alp}$ we have from our construction of $P_\alp$ that $st(P_\alp(a|b))=r$, so for such $\alp$'s we have: 
 $$ \mathcal{M}_\alp\vDash``\forall n\in N\; \left|\frac{p_{a\wedge b}}{p_b}-t_r\right|<1/n"$$  

Now set $\mathcal{M}$ to be the model attained by taking the ultraproduct of the models $M_\alp$ under the ultrafilter $\tilde{\U}$. 
As $\{\alp<\kp:a,b\in\La_{i_\alp}\}$ is in $\tilde{\U}$, \L os' Theorem tells us that the above sentence interpreted in the ultraproduct model will also hold, i.e.
 $$ \mathcal{M}\vDash``\forall n\in N\; \left|\frac{p_{a\wedge b}}{p_b}-t_r\right|<1/n"$$ 
What is the ultraproduct model? The members of the domain are the $\tilde{\U}$ equivalence classes of sequences of objects $\langle a:a\in D_\alp\rangle$, so by the earlier discussion these are exactly the hyperreals generated by $\U$. The interpretation on constants in $\La'$ is straightforward, $\mathfrak{I}(t_r)=[\langle \mathfrak{I}_\alp(t_r):\alp<\kp\rangle]_{\tilde{\U}}$ etc. So again by the earlier discussion we see that $\mathfrak{I}(p_b)=[P_\alp(b)]_{\tilde{\U}}=P(b)$ and $\mathfrak{I}(p_{a\wedge b})=[P_\alp(a\cap b)]_{\tilde{\U}}=P(a\cap b)$. The interpretation of $N$ in the ultraproduct model will be a non-standard extension of the natural numbers, but importantly it will include all of $\nat$ so we can conclude that: $$\forall n\in\nat,\left| \frac{P(a\wedge b)} {P(b)} -r\right|<1/n$$ In other words, $$st(\frac{P(a \wedge b)}{P(b)})=r=C(a,b)$$
So we're done.


\subsection{Discussion}

The upshot is that for every Popper function $C(x,y)$, there is an $NAP$ function $P$ that is pointwise infinitely close to $C$. But in the present context, this leaves considerable latitude for the answer to the question: \emph{how} close is $P$ to $C$?

We have seen that for every $r \in \mathbb{R}$, we can find in $\mathbb{R}^*$ hyperreals that are of different degrees of infinitesimal closeness to $r$. We have shown that the infinitesimal closeness of the $NAP$ function $P$ that we have constructed to the Popper function $C$ that was given is at least pointwise of rank 1. But one might well ask\footnote{Indeed: an anonymous referee did ask.} whether $P$ can be chosen in such a way that its closeness to $C$ is pointwise always of a higher rank. We leave this as a question for further research.


\section{Non-Archimedean probabilities and lexicographic probabilities}

In section \ref{McG} we noted that for finitely generated languages $\mathcal{L}_i$, any value $Pr_i(A)$ of a non-Archimedean probability function can be given a lexicographic representation as a finite sequence of classical real-valued probability values. In the literature on lexicographical probabilities, generalised probabilities are sometimes considered that represent probabilities as \emph{$\omega$-sequences} of classical real-valued probability values.\footnote{For a discussion of the theory of lexicographic probabilities, see \cite{Blume et al 1991}.} The naive conjecture that non-Archimedean probability functions for \emph{(countably) infinitely} generated languages can be represented as an $\omega$-sequence of classical real-valued probability values has been shown to be incorrect \cite[example 4.8, example 4.10]{Halpern 2010}.

The problem is, roughly, one of non-well-foundedness.\footnote{Halpern briefly discusses the idea of non-well-founded lexicographical probability functions in \cite[p.~165]{Halpern 2010}, and dismisses it.} We have seen in section \ref{McG} that the terms in the polynomial expression of $Pr_i(A)$ represents a `level of infinitesimality'. But (as we shall shortly see), the collection of ranks for an infinitely generated language does not in general form a well-ordering. So we will represent $NAP$ probability values lexicographically as \emph{non-well-ordered} sequences of real numbers. We will concentrate on the simplified case of $NAP$ functions that have no associated weight function, or, equivalently, for which the associated weight function is constant 1. 


\subsection{Extending van Fraassen's notion of rank}

First we extend van Fraassen's definition of the notion of rank to $NAP$ functions $P$.

Let $F$ be the non-Archimedean field of which $ran(P)$ is a substructure. Then:
\begin{definition}
$\forall a,b, \in F: a \approx_{rk} b \equiv \exists r \in F:  \infty > \left \vert st[r] \right \vert \neq 0 \textrm{ and } a = r \cdot b$
\end{definition}
So numbers in the field $F$ are of the same rank if they are not infinitely small or infinitely large with respect to each other. 

It is immediate that $\approx_{rk}$ is an equivalence relation. So we define \emph{ranks} as equivalence classes of $\approx_{rk}$:

\begin{definition}
$\forall a \in F: rk(a) \equiv [a]_{\approx_{rk}}$, and $\mathcal{R} \equiv \{ rk(a) : a  \in F  \} .$
\end{definition}
Thus ranks can be seen as \emph{locally Archimedean} substructures of the non-Archimedean field $F$.

$\mathcal{R}$ is a generalisation of van Fraassen's notion of rank to infinitely generated languages. The elements of $\mathcal{R}$ are linearly ordered in a natural way (induced by the linear ordering on $F$): the \emph{higher} the rank of a number, the larger its `degree of infinitesimality', and the probability value $0 \in F$ can be seen as the unique element of $F$ of \emph{maximal rank} $\underline{0}$. But this natural ordering on $\mathcal{R}$ is not in general a well-ordering.


We define the rank $rk(A)$ of an \emph{event} $A \subseteq \Omega$ as $rk(P(A))$. 


Now we arbitrarily choose, for each $\underline{\alpha} \in \mathcal{R}$, a positive \emph{rank unit value} $1_{\underline{\alpha}}$ of rank $\underline{\alpha}$.

Let an $NAP$ function $P$ be given. We want to define a lexicographical representation of $P$. For every  $A \subseteq \Omega$,   we want to define a lexicographial ordering $<_L$ such that  $$ \forall A,B \subseteq \Omega:  P(A) < P(B) \Leftrightarrow P(A) <_L P(B) .$$


\subsection{Transfinite sums of elements of $F$}

The idea is to approximate $P(A)$ by means of a well-ordered (and generally transfinite) sequence of \emph{approximations}. The definition of these approximations involves transfinite sums of elements of $F$. So for our construction we will need a notion of sum of elements of $F$ that makes sense also for all transfinite $\alpha$. We define an appropriate notion of sum using a second ultrafilter construction (recognising that $F$ itself was already generated by an ultrafilter construction).

Let $\mu$ be an ordinal that is chosen (with foresight) to be large enough to enumerate the stages of approximation of elements of $F$.

\begin{definition}
$$ S_{\alpha} \equiv \{ S \in [\mu]^{< \omega} : \alpha \leq \textrm{min}(S)    \}   $$

$$\mathbb{S} \equiv \{ S_{\alpha}: \alpha     \}$$

\end{definition}

Clearly $\mathbb{S}$ has the finite intersection property. So let $U^*$ be an ultrafilter on $[\mu]^{< \omega}$ extending $\mathbb{S}$. Then $U^*$ can be taken to determine an appropriate notion of sum in the following way.

We define sums $\sum_{\alpha < \beta}f(\alpha)$, with $\beta \leq \mu$, inductively. So we assume $\sum_{\alpha < \beta_0}f(\alpha)$ to be defined already for all $\beta_0 < \beta$, and define $\sum_{\alpha < \beta}f(\alpha)$.


\begin{definition}\label{sumdef}
Let $f$ be any function from $\beta$ to $F$. For any $S \in [\beta]^{< \omega}$ with $\beta_0 = \textrm{min}(S)$, let $$ f(S) \equiv 
\sum_{\alpha < \beta_0}f(\alpha) +
 \sum_{\alpha \in S \backslash \{ \beta_0 \}} f(\alpha)  .$$
\end{definition}
Then $f(S)$ is a \emph{finite} sum of elements of $F$, which is of course well-defined because $F$ is a field.

Now we identify modulo agreement on the ultrafilter $U^*$:
\begin{definition}
For any functions $f,g$ from $\beta$ to $F$: $$  f \sim g \equiv \{  S \in [\mu]^{< \omega}: f(S) = g(S) \} \in \mathcal{U}^*   .$$

\end{definition}

This partitions the functions $f$ from $\mu$ to $F$ into equivalence classes $[f]_{\mathcal{U}^*}$, and they form a non-Archimedean field $F^*$ into which $F$ is canonically embedded in the same way as $\mathbb{R}$ is embedded in $F$. 

Now we set:

\begin{definition}
$$ \sum_{\alpha < \beta}f(\beta) \equiv \lceil [f]_{\mathcal{U}^* }\rceil_F ,$$
\end{definition}
where $\lceil a^* \rceil_F$ is the unique element $a\in F$ that is closest to the element $a^* \in F^*$.

Not all such sums will be well-defined. Intuitively, it may be the case that $\mid [f]_{\mathcal{U}^*}\mid$ is ``infinitely large'' with respect to all elements of $F$, in the same way that some elements of $F$ are ``infinitely large'' with respect to all elements of $\mathbb{R}$.  But if $[f]_{\mathcal{U}^*}$ is bounded from below and from above by elements of $F$ as canonically embedded in $F^*$, then the sum \emph{is} well-defined, because of transfer. The argument goes as follows. Using the $*$-notation from non-standard analysis, we move from the $\mathbb{R}$ to the non-Archimedean field $\mathbb{R}^* = F$ and then to $\mathbb{R}^* = F^*$. Then using this $*$-notation, we know from the completeness of the real field that 
$$ \forall r^* \in \mathbb{R}^* \exists r \in \mathbb{R} \forall s \in \mathbb{R}: s \neq r \rightarrow \mid r^* - s \mid > \mid r^* - r \mid .  $$ By $*$-transfer, this yields: $$ \forall r^{**} \in \mathbb{R}^{**} \exists r^* \in \mathbb{R^*} \forall s^* \in \mathbb{R^*}: s^* \neq r^* \rightarrow \mid r^{**} - s^* \mid > \mid r^{**} - r^* \mid .  $$






\subsection{Approximations}\label{apx}

Now we are ready to define the approximations. Let a $NAP$ function $P$ and an event $A \subseteq \Omega$ be given.

The idea is to approximate $P(A)$ as follows. $P(A)$ is of a certain rank $rk(A)$: call this rank $\underline{0}^A$ (``the rank of the $0$th approximation of $A$"). So we will in a first stage approximate $P(A)$ by the element $A_0 \cdot 1_{\underline{0}^A}$ where $A_0 \in \mathbb{R}$ is such that $A_0 \cdot 1_{\underline{0}^A}$ is closest to $P(A)$. 
But then it is likely that there is a non-zero \emph{remainder}  $P(A) -  A_0 \cdot 1_{\underline{0}^A}$. This remainder will then be of a \emph{higher} rank $\underline{1}^A$ than $\underline{0}^A$. So we will approximate the remainder by the element of the form $A_{\underline{1}^A} \cdot 1_{\underline{1}^A}$, with $A_{\underline{1}^A}$ the unique real number such that that $A_{\underline{1}^A} \cdot 1_{\underline{1}^A}$ is closest to the remainder. This will leave us with a remainder of a still higher rank. Thus we continue into the transfinite. 

The details of the construction go as follows.

We define the remainders and the approximating real numbers inductively. We take them to be defined for all $\alpha < \beta$, and then first define the remainder of stage $\beta$:

\begin{definition}\label{remainderdef}

$$  A_{\beta}^r \equiv P(A) - \sum_{\alpha < \beta} A_{\alpha} \cdot 1_{\underline{\alpha}^A}  $$

\end{definition} 

This remainder will be of a certain rank:
\begin{definition}
$$1_{\underline{\beta}^A} \equiv  rk(A^r_{\beta} ) $$

\end{definition}

On the basis of this, we then define the \emph{approximation} of $P(A)$ at stage $\beta$:
\begin{definition}
$A_{\beta} $ is the unique $r \in \mathbb{R}$ such that $r \cdot 1_{\underline{\beta}^A}$ is closer to $  A_{\beta}^r$ than any number of the form $s \cdot 1_{\underline{\beta}^A}$ for $s \in \mathbb{R}$ such that $s \neq r$
\end{definition}

Of course these definitions are only well-formed if the sums involved (definition \ref{remainderdef}) are well-defined. But this is the case:

\begin{proposition}
For all $\beta < \mu$, $A_{\beta}^r$ is well-defined.

\begin{proof}
By transfinite induction on $\beta$. We use boundedness considerations that hold for ultrafilter-large families of finite sets and are then globally preserved.
\end{proof}
\end{proposition}

\begin{proposition}\label{succrkprop}

$ A^r_{\beta + 1}  = A^r_{\beta} - A_{\beta} \cdot 1_{\underline{\beta}^A}   $

\begin{proof}
The reason is that $S_{\beta} \in \mathcal{U}^*$ (degenerate case).
\end{proof}
\end{proposition}

\begin{lemma}\label{vanish}
If $rk(A^r_{\alpha}) \neq \underline {0}$, then $\alpha < \beta \leq \mu \Rightarrow rk(A^r_{\alpha}) < rk(A^r_{\beta}).$

\begin{proof}
Induction on $\beta$.

\noindent 1. $\beta = \gamma + 1$. By the previous proposition and the definition of $A_{\gamma}$ we have $rk(A^r_{\gamma + 1}) > rk(A^r_{\gamma})$. The result then follows by the induction hypothesis.

\noindent 2. $Lim(\beta)$. Let $\alpha < \beta$. Then by the induction hypothesis we have $rk(A^r_{\alpha}) < rk(A^r_{\alpha + 1})$. So it suffices to show that  $rk(A^r_{\beta}) \geq rk(A^r_{\alpha + 1})$. 

Since $S_{\alpha +1} \in \mathcal{U}^*$, we want to show that for each $S \in S_{\alpha +1}$, we have $$rk(P(A) - \sum_{\gamma \in S}A_{\gamma} \cdot 1_{\underline{\gamma}^A}) \geq rk(P(A) - \sum_{\gamma < \alpha + 1}A_{\gamma} \cdot 1_{\underline{\gamma}^A}) .$$ 

We only need to look at those $S$ such that $\textrm{min}(S)< \beta$, since only they can contribute to the sums. By the induction hypothesis, for such $S$ we have $$rk(P(A) - \sum_{\gamma < \textrm{min}(S)}A_{\gamma} \cdot 1_{\underline{\gamma}^A}) \geq rk(P(A) - \sum_{\gamma < \alpha + 1}A_{\gamma} \cdot 1_{\underline{\gamma}^A}) .  $$ This entails that indeed $$rk(P(A) - \sum_{S}A_{\gamma} \cdot 1_{\underline{\gamma}^A}) \geq rk(P(A) - \sum_{\gamma < \alpha + 1}A_{\gamma} \cdot 1_{\underline{\gamma}^A}) .  $$
\end{proof}
\end{lemma}

So there must for simple cardinality reasons be an ordinal $\zeta$ such that $rk(A^r_{\zeta}) = \underline{0}$. Then $\zeta$ is called the \emph{closure ordinal} for $A$. In a similar vein, the closure ordinal of $P$, denoted as $cl(P)$,  is defined as $$max\{\alpha : \alpha \textrm{ is the closure ordinal of some } A \subseteq \Omega \} \footnote{In fact, we have a notion of closure ordinal for each $r \in F$. So we could also define the closure of $P$ as $max\{\alpha : \alpha \textrm{ is the closure ordinal of some } r \in F\}$.}$$ This means that the $NAP$ functions on $\Omega$ yield an \emph{ordinal spectrum} that is determined by their closure ordinals.

\begin{theorem}\label{sumtheorem}

$$ P(A) = \sum_{\alpha < cl(A)} A_{\alpha} \cdot 1_{\underline{\alpha}^A} $$

\begin{proof}
Consider the closure ordinal $\zeta$ of $A$, which must exist by lemma \ref{vanish}. Then we have $$P(A) -  \sum_{\beta < \zeta}A_{\beta} \cdot 1_{\underline{\beta}^A} = 0. $$
\end{proof}
\end{theorem}
Of course we can then also express $ P(A)$ as  $ \sum_{\alpha < cl(P)} A_{\alpha} \cdot 1_{\underline{\alpha}^A} .$

The infinite sum $\sum_{\alpha < cl(A)} A_{\alpha} \cdot 1_{\underline{\alpha}^A} $ can be seen as a kind of \emph{Cantor normal form} for $P(A)$.\footnote{In a somewhat related (but also significantly different) context, the connection between a non-Archimedean notion of size and Cantor normal forms is explored in \cite[section 1.5]{Benci et al 2006}.}
The upshot of our discussion is that even though the structure $\mathcal{R}$ of the ranks is non-wellfounded, every probability value can be expressed as a well-founded infinite sum of components of increasing rank.


\subsection{Representation theorem}\label{rep2}

The infinite sum $\sum_{\alpha < cl(A)} A_{\alpha} \cdot 1_{\underline{\alpha}^A}$  can be seen as a \emph{lexicographic presentation} of $P(A)$. It can be used to define a lexicographical ordering:
\begin{definition}[lexicographic order]
$P(A) <_L P(B) \equiv$
for the smallest $\alpha$ such that $A_{\alpha} \cdot 1_{\underline{\alpha}^A} \neq B_{\alpha} \cdot 1_{\underline{\alpha}^B}$, we have $A_{\alpha} \cdot 1_{\underline{\alpha}^A} < B_{\alpha} \cdot 1_{\underline{\alpha}^B}$
\end{definition}
It is immediate that $<_L$ is a strict linear ordering.

Now we are ready to prove the main theorem of this section:
\begin{theorem}[representation theorem]\label{rep}
$$ \forall A,B \subseteq \Omega: P(A) < P(B)  \Leftrightarrow  P(A) <_L P(B) . $$

\begin{proof}

We know from theorem \ref{sumtheorem} that $P(A) = \sum_{\alpha < cl(P)} A_{\alpha} \cdot 1_{\underline{\alpha}^A} $   and $ P(B) =  \sum_{\alpha < cl(P)} B_{\alpha} \cdot 1_{\underline{\alpha}^B}  .$

\noindent ($\Rightarrow$) We are given that $P(A) < P(B)$. Let $\alpha$ be the first ordinal where $P(A)$ and $P(B)$ differ. Then we want to show that $$A_{\alpha} \cdot 1_{\underline{\alpha}^A} < B_{\alpha} \cdot 1_{\underline{\alpha}^B} .$$ 

Suppose, for a contradiction, that this is not the case, i.e., that $$A_{\alpha} \cdot 1_{\underline{\alpha}^A} > B_{\alpha} \cdot 1_{\underline{\alpha}^B} .$$ We will show that then, for all $S \in S_{\alpha}$:
$$  \sum_S  A_{\alpha} \cdot 1_{\underline{\alpha}^A} > \sum_S  B_{\alpha} \cdot 1_{\underline{\alpha}^B}.   $$

We aim to show this by an induction on $\textrm{min}(S)$.

\noindent 1. Suppose $\alpha = \textrm{min}(S)$. Then the property holds by lemma \ref{vanish}. 

\noindent 2. Suppose the property holds for all  $S \in S_{\alpha}$ such that $\textrm{min}(S) < \beta$. Then we want to show that the property also holds for all $S \in S_{\alpha}$ such that $\textrm{min}(S) = \beta$.

\noindent 2a. Suppose $\beta= \gamma +1$. For $A$, we know that $$ \sum_{\kappa < \beta} A_{\kappa} \cdot 1_{\underline{\kappa}^A} = \sum_{\kappa < \gamma} A_{\kappa} \cdot 1_{\underline{\kappa}^A} + A_{\gamma} \cdot 1_{\underline{\gamma}^A},$$ and similarly for $B$. So the property follows by the induction hypothesis and lemma \ref{vanish}.

\noindent 2b. $Lim(\beta)$. It suffices to show $$  \sum_{\kappa < \beta}  A_{\kappa} \cdot 1_{\underline{\kappa}^A} > \sum_{\kappa < \beta}  B_{\kappa} \cdot 1_{\underline{\kappa}^B}.$$    But by the induction hypothesis, this holds on all $S \in S_{\alpha}$, and $S_{\alpha} \in \mathcal{U}^*$, so this indeed holds also.

From this inductive argument we conclude that for all $S \in S_{\alpha}$, $$  \sum_S  A_{\alpha} \cdot 1_{\underline{\alpha}^A} > \sum_S  B_{\alpha} \cdot 1_{\underline{\alpha}^B}.   $$ So, since $S_{\alpha} \in \mathcal{U}^*$, $$ P(A) = \sum_{\alpha < cl(P)} A_{\alpha} \cdot 1_{\underline{\alpha}^A} > \sum_{\alpha < cl(P)} B_{\alpha} \cdot 1_{\underline{\alpha}^B} = P(B), $$ which gives us the required contradiction.

\noindent ($\Leftarrow$) This follows by a similar argument. It is given that $P(A) <_L P(B)$. Then there is a first $\alpha$ such that $$A_{\alpha} \cdot 1_{\underline{\alpha}^A} < B_{\alpha} \cdot 1_{\underline{\alpha}^B} .$$ Then we argue inductively that $$  \sum_{\alpha < cl(P)} A_{\alpha} \cdot 1_{\underline{\alpha}^A} < \sum_{\alpha < cl(P)} B_{\alpha} \cdot 1_{\underline{\alpha}^B} , $$
and we are done.
\end{proof}
\end{theorem}
Note that this implies that even though a choice of ultrafilter $\mathcal{U}^*$ was needed to define $<_L$, the resulting ordering is invariant with respect to this choice of ultrafilter.


\subsection{Discussion}

Let a sample space $\Omega$ be given. Then an ultrafilter on $\mathcal{P}_{fin}(\Omega)$ determines a $NAP$ function $P$ (with `uniform weight 1'). Different such $NAP$ functions may have different closure ordinals in the sense of section \ref{apx}. Intuitively, $NAP$ functions with larger closure ordinals may be taken to be `more complicated' than $NAP$ functions with smaller closure ordinals. Thus analysing the closure ordinals of $NAP$ functions would yield a \emph{classification} or \emph{spectral analysis} of $NAP$ functions (with uniform weight 1, on a fixed sample space). Thus it might be a worthwhile project to undertake a spectral analysis of $NAP$ functions.

The representation theorem that was proved in the previous section (theorem \ref{rep}) is related to a very general representation theorem that is announced in \cite[section 6]{Pedersen forthc.}. In this article, Pedersen extends De Finetti's fundamental theorem of comparative expectations to expectation orderings that satisfy a version of the \emph{principle of weak dominance} (rather than uniform simple dominance). He then announces that such systems of comparative expectations (finite and infinite) can be represented by an expectation function that takes values in a non-Archimedean field in which in which
every number can be written as formal well-founded power series in a single infinitesimal. This means, of course, that the elements in such a field are lexicographically ordered.


\section{Closing}

Representation theorems, like those of McGee \cite{McGee}, Van Fraassen \cite{VanFraassen}, Leitgeb \cite{Leitgeb}, and the two representation theorems in the present article, give us reasons to think that Popper functions capture a robust concept of probability. Popper functions cannot do anything that $NAP$ functions and lexicographical probability functions can't do. So if we believe any of these model our intuitive notion of probability well, then Popper functions must too. Though they may not fully cash out all our intuitions, for example about infinite additivity, they cannot go against them to any greater extent than these alternatives. 

Also, these representation theorems, which relate all the different models via Popper functions, give us reasons to believe that they all correctly represent our intuitions of probability, for the same reason that the equivalence of the different mathematical attempts at describing what an algorithm is gives us reason to believe we have captured that notion: there are many ways to be wrong about something, so it would seem unlikely that every time we attempted it we were wrong \emph{in the same way}. 


So we have better reasons to believe that each of these theories for probability really do model an intuitive concept. In particular, theories in other areas of philosophy which use Popper functions as representing probabilities gain more support from this representation as $NAP$ functions. 
$NAP$ gives us both a notion of how probabilities for events depend on the probabilities of the individual outcomes that make up those events, and also relates the conditional probability to the ratio of absolute probabilities in the familiar, intuitive way. So Popper functions are perhaps closer to our intuitions concerning probability than their initial axiomatic presentation may suggest.

But non-Archimedean probability functions in general, and $NAP$ functions in particular, are themselves not completely intuitive. One source of un-intuitiveness is the non-wellfoundedness of the degrees of infinitesimality that such probability functions entail. This seems to open the prospect that the possibility of one event can be smaller than that of another event even though there is no largest degree of infinitesimality at which they differ. Nonetheless, this turns out not to be the case. Our second representation theorem (section \ref{rep2}) shows that every $NAP$ function value can be represented as a well-founded power series of which each term represents the contribution of a specific degree of infinitesimality.

In sum, our conclusion is that Popper functions, $NAP$ functions, and lexicographical functions cohere well, not only on finite sample spaces, but also on infinite sample spaces.



\bigskip



\newpage
\begin{thebibliography}{ZZ}

\bibitem[Benci et al 2006]{Benci et al 2006} Benci, V., Di Nasso, M, Forti, M. \textsl{An Aristotelian notion of size.} Annals of Pure and Applied Logic \textbf{143}(2006), p.~43--53.

\bibitem[Benci et al 2013] {NAP} Benci, V.; Horsten, L., \& Wenmackers, S. \textsl{Non-Archimedean Probability}, Milan Journal of Mathematics \textbf{81}(2013), p.~121--151.  

\bibitem[Benci et al forthc.]{BHW forthc} Benci, V.; Horsten, L., \& Wenmackers, S. \textsl{Infinitesimal probabilities.} Forthcoming in the British Journal for Philosophy of Science.

\bibitem[Blume et al 1991]{Blume et al 1991} Blume, L., Brandenburger, A., \& Dekel, E. \textsl{Lexicographical probabilities and choice under uncertainty.} Econometrica \textbf{59}(1991), p.~61--79.

\bibitem[Goldblatt 1998]{hyperreals} Goldblatt, R. \emph{Lectures on the hyperreals. An introduction to nonstandard analysis.} Graduate Texts in Mathematics, 188. Springer-Verlag, New York, 1998.

\bibitem[Hajek 2003]{Hajek 2003} Hajek, A. \textsl{What conditional probability could not be.} Synthese \textbf{137}(2003), p.~273--323.

\bibitem[Halpern 2010]{Halpern 2010} Halpern, J. \textsl{Lexicographic probability, conditional probability, and non-standard probability.} Games and Economic Behavior \textbf{68}(2010), p.~155--179.

\bibitem[Harper 1975]{Harper 1975} Harper, W. \textsl{Rational Belief Change, Popper Functions and Counterfactuals.} Synthese \textbf{30}(1975), p.~221--262.

\bibitem[Kolmogorov 1933]{Kolmogorov} Kolmogorov, A. \emph{Grundbegriffe der Wahrscheinlichkeitrechnung (Ergebnisse Der Mathematik)} (1933). Translated by Morrison, N. \emph{Foundations of probability.} Chelsea Publishing Company (1956) 2nd English edition.


\bibitem[Leitgeb 2012]{Leitgeb} Leitgeb, H. \textsl{A Probabilistic Semantics For Counterfactuals},
 The Review of Symbolic Logic,  \textbf{5}(2012), p.~26--121.
 
 \bibitem[Levi 1989]{Levi 1989} Levi, I. \textsl{Possibility and probability.} Erkenntnis, \textbf{31}(1989), p.~365Ð-386.

\bibitem[McGee 1994]{McGee} McGee, V. \textsl{Learning the impossible}, in Ellery Eells \& Brian Skyrms (eds.), \textsl{Probability and Conditionals: Belief Revision and Rational Decision}, Cambridge University Press, p.~179--199.

\bibitem[Nelson 1987]{Nelson 1987} Nelson, E. \textsl{Radically elementary probability theory.} Princeton University Press, 1987.

\bibitem[Pedersen forthc.]{Pedersen forthc.} Pedersen, A. \textsl{Comparative expectations.} Forthcoming in Studia Logica.

\bibitem[Popper 1959]{Popper} Popper, K. \emph{The Logic of Scientific Discovery.} Basic Books, 1959.

\bibitem[Pruss 2015]{Pruss 2015} Pruss, A. \textsl{Popper functions, uniform distributions, and infinite sequences.} Journal of Philosophical Logic \textbf{44}(2015), p.~259--271.

\bibitem[Robinson 1961]{Robinson 1961} Robinson, A. \textsl{Non-standard analysis}, Nederl. Acad. Wetensch. Proc. Ser. A \textbf{64} and Indag. Math. \textbf{23}(1961), pp.~432--440.

\bibitem[Seidenfeld et al 2014]{Seidenfeld et al 2014} Seidenfeld, T., Schervish, M., Kadane, J. \textsl{Non-conclomerability for countably additive measures that are not $\kappa$-additive.} Unpublished manuscript.

\bibitem[Van Fraassen 1976]{VanFraassen} Van Fraassen, B. \textsl{Representation of Conditional Probabilities}, Journal of Philosophical Logic \textbf{5}(1976), p.~417--430. 

\bibitem[Wenmackers and Horsten 2013]{W&H} Wenmackers, S., Horsten, L. \textsl{Fair infinite lotteries}, Synthese \textbf{190}(2013), p.~37--61.

\end{thebibliography}
\end{document}